\documentclass[oneside,english]{amsart}
\usepackage[T1]{fontenc}
\usepackage[latin9]{inputenc}
\usepackage{amsthm}
\usepackage{amsbsy}
\usepackage{amstext}
\usepackage{amssymb}
\usepackage{esint}
\usepackage[numbers]{natbib}

\makeatletter
\numberwithin{equation}{section}
  \theoremstyle{definition}
  \newtheorem{defn}{\protect\definitionname}
  \theoremstyle{plain}
  \newtheorem{lem}{\protect\lemmaname}
\theoremstyle{plain}
\newtheorem{thm}{\protect\theoremname}
  \theoremstyle{remark}
  \newtheorem{rem}{\protect\remarkname}
  \theoremstyle{remark}
  \newtheorem*{acknowledgement*}{\protect\acknowledgementname}

\makeatother

\usepackage{babel}
  \providecommand{\acknowledgementname}{Acknowledgement}
  \providecommand{\definitionname}{Definition}
  \providecommand{\lemmaname}{Lemma}
  \providecommand{\remarkname}{Remark}
\providecommand{\theoremname}{Theorem}

\begin{document}

\title{The Regularity of some Vector-Valued Variational Inequalities with
Gradient Constraints}

\author{Mohammad Safdari}
\begin{abstract}
We prove the optimal regularity for some class of vector-valued variational
inequalities with gradient constraints. We also give a new proof for
the optimal regularity of some scalar variational inequalities with
gradient constraints. In addition, we prove that some class of variational
inequalities with gradient constraints are equivalent to an obstacle
problem, both in the scalar and vector-valued case.%
\thanks{Email address: safdari@berkeley.edu%
}
\end{abstract}

\maketitle

\section{Introduction}

Let $U\subset\mathbb{R}^{n}$ be an open bounded set. Suppose $K\subset\mathbb{R}^{n}$
is a balanced (symmetric with respect to the origin) compact convex
set whose interior contains $0$. Also suppose that $\boldsymbol{\eta}\in\mathbb{R}^{N}$
is a fixed nonzero vector. Consider the following problem of minimizing
\begin{equation}
I(\mathbf{v}):=\int_{U}|D\mathbf{v}|^{2}-\boldsymbol{\eta}\cdot\mathbf{v}\, dx
\end{equation}
over 
\begin{equation}
K_{1}:=\{\mathbf{v}=(v^{1},\cdots,v^{N})\in H_{0}^{1}(U;\mathbb{R}^{N})\,\mid\,\|D\mathbf{v}\|_{2,K}\leq1\textrm{ a.e.}\},
\end{equation}
Where 
\begin{equation}
\|A\|_{2,K}:=\underset{z\ne0}{\sup}\,\frac{|Az|}{\gamma_{K}(z)}
\end{equation}
for an $N\times n$ matrix $A$, and $\gamma_{K}$ is the norm associated
to $K$ defined by 
\begin{equation}
\gamma_{K}(x):=\inf\{\lambda>0\,\mid\, x\in\lambda K\}.
\end{equation}
As $K_{1}$ is a closed convex set and $I$ is coercive, bounded and
weakly sequentially lower semicontinuous, this problem has a unique
solution $\mathbf{u}$. We will show that under some extra assumptions
on $K$ 
\[
\mathbf{u}\in C_{\textrm{loc}}^{1,1}(U;\mathbb{R}^{N}).
\]

This problem is a generalization to the vector-valued case of the
elastic-plastic torsion problem, which is the problem of minimizing
\[
J_{\eta}(v):=\int_{U}|Dv|^{2}-\eta v\, dx
\]
for some $\eta>0$, over 
\[
\{v\in H_{0}^{1}(U)\,\mid\,|Dv|\le1\textrm{ a.e.}\}.
\]
The regularity of the elastic-plastic torsion problem has been studied
by \citet{MR0239302}, and \citet{MR513957}. There has been several
extensions of their results to more general scalar problems with gradient
constraints. See for example \citet{MR697646}, \citet{MR0385296},
\citet{MR529814}, \citet{MR607553}, \citet{MR693645}. To the best
of author's knowledge, the only work on the regularity of vector-valued
problems with gradient constraints is \citet{MR1205846}.

Our approach is to show that the above vector-valued problem is reducible
to the scalar problem of minimizing $J_{1}$ over 
\[
\{v\in H_{0}^{1}(U)\,\mid\,|\boldsymbol{\eta}|Dv\in K^{\circ}\textrm{ a.e.}\},
\]
where $K^{\circ}$ is the polar of $K$ (See section 2). Then we show
that this scalar problem is equivalent to a double obstacle problem
with only Lipschitz obstacles. At the end, we generalize the proof
of \citet{MR513957}, to obtain the optimal regularity. We should
note that \citet{MR1875900} proves the regularity of a more general
double obstacle problem by different methods.

In the process described above, we also show that our vector-valued
problem with gradient constraint is equivalent to a vector-valued
obstacle problem. This result, which is the first result of its kind
as far as the author knows, is a generalization to the vector-valued
case of the equivalence between the elastic-plastic torsion problem
and an obstacle problem, proved by \citet{MR0346345}. Later \citet{MR1797872}
proved that the equivalence holds for a larger class of scalar variational
inequalities with gradient constraints. We will further generalize
their result. Suppose $f:\mathbb{R}^{n}\rightarrow\mathbb{R}$ and
$g:\mathbb{R}\rightarrow\mathbb{R}$ are convex functions. Consider
the problem of minimizing 
\begin{equation}
J(v):=\int_{U}f(Dv(x))+g(v(x))\, dx
\end{equation}
over
\begin{equation}
W_{K}:=\{v\in u_{0}+W_{0}^{1,p}(U)\,\mid\, Dv(x)\in K\;\textrm{a.e.}\},
\end{equation}
where $u_{0}\in W^{1,p}(U)$. We will show that under appropriate
assumptions, the minimizer of $J$ over $W_{K}$ is the same as its
minimizer over 
\begin{equation}
W_{u^{-},u^{+}}:=\{v\in u_{0}+W_{0}^{1,p}(U)\,\mid\, u^{-}(x)\leq v(x)\leq u^{+}(x)\;\textrm{a.e.}\},
\end{equation}
for some suitable functions $u^{-},u^{+}$. The difference of our
result with that of \citet{MR1797872} is that we allow $f,g$ to
be only convex, and $K$ to have empty interior. Some of our results
has been proved using different means by \citet{MR1881695}.

\section{The Equivalence in the Scalar Case}

Suppose $K\subset\mathbb{R}^{n}$ is a compact convex set whose interior
contains the origin. Let $J$, $W_{K}$, and $W_{u^{-},u^{+}}$ be
as above. We assume that on $W^{1,p}(U)$, $J$ is finite, bounded
below and sequentially weakly lower semicontinuous. These assumptions
are satisfied if, for example, we impose some growth conditions on
$f,g$ and some mild regularity on $\partial U$. Therefore by our
assumption, $J$ attains its minimum on any nonempty closed convex
subset of $W^{1,p}(U)$. 

Furthermore, we assume that $u_{0}$ is Lipschitz, and 
\[
Du_{0}\in K\qquad\textrm{ a.e.}.
\]
Thus in particular, $W_{K}$ is nonempty.
\begin{defn}
The \textbf{gauge} of $K$ is a convex function defined by 
\begin{equation}
\gamma_{K}(x):=\inf\{\lambda>0\,\mid\, x\in\lambda K\},
\end{equation}
and its \textbf{polar} is the convex set 
\begin{equation}
K^{\circ}:=\{x\,\mid\, x\cdot k\leq1\,\textrm{ for all }k\in K\}.
\end{equation}

\end{defn}
We recall that for all $x,y\in\mathbb{R}^{n}$, we have 
\begin{equation}
x\cdot y\leq\gamma_{K}(x)\gamma_{K^{\circ}}(y).
\end{equation}
Its proof can be found in \citet{MR0274683}. Also, when $K$ is balanced,
$K^{\circ}$ is balanced too, and $\gamma_{K},\gamma_{K^{\circ}}$
are both norms on $\mathbb{R}^{n}$.

Now, let us find $u^{\pm}\in W_{K}$ such that for all $u\in W_{K}$
we have $u^{-}\leq u\leq u^{+}$. Let $u^{\pm}$ be respectively the
unique minimizers of $J^{\pm}(v)=\int_{U}\mp v(x)\, dx$ over $W_{K}$.
We show that they have the desired property. We need the following
lemma.
\begin{lem}
\label{Lemma bdd grad}Suppose $u$ is a compactly supported function
in $W^{1,p}(\mathbb{R}^{n})$ with $Du\in K$ a.e.. Then 
\begin{equation}
u(y)-u(x)\leq\gamma_{K^{\circ}}(y-x)
\end{equation}
 for all $x,y$.\end{lem}
\begin{proof}
Consider the mollifications 
\[
u_{\epsilon}(x):=(\eta_{\epsilon}\star u)(x):=\int_{B_{\epsilon}(x)}\eta_{\epsilon}(x-y)u(y)\, dy,
\]
where $\eta_{\epsilon}$ is a nonnegative smooth function with support
in $B_{\epsilon}(0)$, and $\int_{B_{\epsilon}(0)}\eta_{\epsilon}\, dx=1$.
Then we know that $u_{\epsilon}$ converges to $u$ a.e., and $Du_{\epsilon}=\eta_{\epsilon}\star Du$.
Hence 
\begin{eqnarray*}
 & \gamma_{K}(Du_{\epsilon}(x)) & \leq\int_{B_{\epsilon}(x)}\gamma_{K}(\eta_{\epsilon}(x-y)Du(y))\, dy\\
 &  & =\int_{B_{\epsilon}(x)}\eta_{\epsilon}(x-y)\gamma_{K}(Du(y))\, dy\;\le1,
\end{eqnarray*}
where we used Jensen's inequality in the first inequality. Thus

\begin{eqnarray*}
 & u_{\epsilon}(y)-u_{\epsilon}(x) & =\int_{0}^{1}Du_{\epsilon}(x+t(y-x))\cdot(y-x)\, dt\\
 &  & \leq\int_{0}^{1}\gamma_{K}(Du_{\epsilon}(x+t(y-x)))\gamma_{K^{\circ}}(y-x)\, dt\;\le\gamma_{K^{\circ}}(y-x).
\end{eqnarray*}
Now we can let $\epsilon\rightarrow0$ to obtain 
\[
u(y)-u(x)\leq\gamma_{K^{\circ}}(y-x)\qquad\qquad\textrm{ for a.e. }x,y.
\]
We can redefine $u$ on the measure zero set where this relation fails,
in a similar way that we extend Lipschitz functions to the closure
of their domains. The extension will satisfy this relation everywhere.\end{proof}
\begin{lem}
Each function in $W_{K}$ is Lipschitz continuous. Also, $W_{K}$
is bounded in $L^{\infty}(U)$ and in $W^{1,p}(U)$. \end{lem}
\begin{proof}
To see this, let $u\in W_{K}$. Then $u=u_{0}+v$ where $v\in W_{0}^{1,p}(U)$.
Thus 
\[
|Dv|=|Du-Du_{0}|<2R
\]
for some $R>0$. Now we can extend $v$ by zero to all of $\mathbb{R}^{n}$,
and the extension will satisfy the same gradient bound. Therefore
by arguments similar to the previous lemma, we can see that the extension
of $v$, and hence $v$ itself, is Lipschitz with Lipschitz constant
$2R$. Using the fact that $v$ is zero on the boundary, this also
implies that $\Vert v\Vert_{L^{\infty}}\leq2RD$, where $D$ is the
diameter of $U$. The result for $u$ follows easily, noting that
$u_{0}$ is Lipschitz.

Now as $\Vert Du\Vert_{L^{\infty}}<C$ for some constant $C$ independent
of $u$, we have $\Vert Du\Vert_{L^{p}}<C$ since $U$ is bounded.
Noting that all $u\in W_{K}$ have the same boundary value, we get
by Poincare inequality $\Vert u\Vert_{W^{1,p}}<C$.
\end{proof}
Now we can see that $J^{\pm}$ are bounded on $W_{K}$. As $J^{\pm}$
are linear, they are weakly continuous. Furthermore $W_{K}$ is convex,
closed and bounded in $W^{1,p}(U)$. Hence $W_{K}$ is compact with
respect to sequential weak convergence. These imply that $J^{\pm}$
have minimizers over $W_{K}$. The uniqueness and the fact that $u^{-}\leq u^{+}$
a.e. on $U$, follows from a similar argument to the proof of the
next lemma.
\begin{lem}
We have 
\[
W_{K}\subset W_{u^{-},u^{+}}.
\]
\end{lem}
\begin{proof}
Suppose $u\in W_{K}$, then $J^{\pm}(u^{\pm})\leq J^{\pm}(u)$. Thus
\[
\int_{U}\mp u^{\pm}\, dx\leq\int_{U}\mp u\, dx,
\]
so 
\[
\int_{U}u^{-}\, dx\leq\int_{U}u\, dx\leq\int_{U}u^{+}\, dx.
\]
Suppose to the contrary that, for example, the set $E:=\{x\,\mid\, u(x)>u^{+}(x)\}$
has positive measure. Consider the function 
\[
w(x):=\max(u,u^{+})=\begin{cases}
u^{+}(x) & x\notin E\\
u(x) & x\in E.
\end{cases}
\]
The derivative of $w$ is 
\[
Dw(x)=\begin{cases}
Du^{+}(x) & x\notin E\\
Du(x) & x\in E
\end{cases}\qquad\textrm{ for a.e. }x.
\]
Therefore we have $Dw(x)\in K$ a.e.. Thus 
\[
J^{+}(w)=-\int_{U}w\, dx<-\int_{U}u^{+}\, dx=J^{+}(u^{+}),
\]
which is a contradiction.
\end{proof}
The following characterization of $u^{\pm}$ will be used later. Here
$d_{K^{\circ}}$ is the metric associated to the norm $\gamma_{K^{\circ}}$.
\begin{thm}
Suppose $u_{0}$ equals a constant $c$ everywhere. Then 
\[
u^{\pm}(x)=c\pm d_{K^{\circ}}(x,\partial U).
\]
\end{thm}
\begin{proof}
It is enough to show that $c\pm d_{K^{\circ}}(x,\partial U)$ are
the minimizers of $J^{\pm}$. The fact that $c\pm d_{K^{\circ}}(x,\partial U)$
belong to $W_{K}$ is equivalent to the fact that $d_{K^{\circ}}(x,\partial U)$
is in $W_{0}^{1,p}(U)$ and its derivative has $\gamma_{K}$ norm
less than one. But $d_{K^{\circ}}(x,\partial U)$ is a Lipschitz function
that vanishes on the boundary of $U$. It also satisfies 
\[
d_{K^{\circ}}(x,\partial U)-d_{K^{\circ}}(y,\partial U)\leq\gamma_{K^{\circ}}(x-y).
\]
As proved by \citet{MR1797872}, this last property implies that the
$\gamma_{K}$ norm of the derivative of $d_{K^{\circ}}(x,\partial U)$
is less than or equal to $1$ a.e..

Now similarly to the proof of Lemma \ref{Lemma bdd grad}, we can
show that 
\[
|v(x)-c|\le d_{K^{\circ}}(x,\partial U)
\]
for all $v\in W_{K}$. Therefore $c\pm d_{K^{\circ}}(x,\partial U)$
minimize $J^{\pm}$ over $W_{K}$.
\end{proof}
The following theorem is the generalization of the result of \citet{MR1797872}.
We removed the assumptions on the derivatives of $g$, and allowed
$K$ to have empty interior.
\begin{thm}
\label{Thm equiv 2}Suppose $K$ is a compact convex set containing
$0$, and $u_{0}$ is the restriction to $U$ of a compactly supported
function in $W^{1,p}(\mathbb{R}^{n})$ with gradient a.e. in $K$.
Also, suppose $f,g$ are convex and at least one of them is strictly
convex. Then the minimizer of 
\[
J(v)=\int_{U}f(Dv(x))+g(v(x))\, dx
\]
over $W_{u^{-},u^{+}}$ is the same as its minimizer over $W_{K}$.\end{thm}
\begin{proof}
Note that the convexity assumptions on $f,g$ imply that the minimizer
of $J$ over any nonempty convex closed set is unique. Also the assumption
on $u_{0}$ implies $u_{0}(y)-u_{0}(x)\leq\gamma_{K^{\circ}}(y-x)$
for all $x,y\in U$, by Lemma \ref{Lemma bdd grad}. Let the minimizer
of $J$ over $W_{u^{-},u^{+}}$ be $u$. As $W_{K}\subset W_{u^{-},u^{+}}$,
it is enough to show that $u\in W_{K}$. 

First assume that $0$ is in the interior of $K$, and $g$ is $C^{1}$
with strictly increasing derivative. 

Similarly to \citet{MR1797872}, using $u_{0}$ we can extend $u^{\pm}$
and $u$ to all of $\mathbb{R}^{n}$ in a way that the gradient of
$u^{\pm}$ is still in $K$. Fix a nonzero vector $h\in\mathbb{R}^{n}$,
and define
\[
\begin{array}{c}
u_{h}^{+}(x):=\max\{u(x+h)-\gamma_{K^{\circ}}(h),u(x)\}\\
u_{h}^{-}(x):=\min\{u(x-h)+\gamma_{K^{\circ}}(h),u(x)\},
\end{array}
\]
and 
\[
\begin{array}{c}
E^{+}:=\{x\in\mathbb{R}^{n}\,\mid\, u_{h}^{+}(x)=u(x+h)-\gamma_{K^{\circ}}(h)>u(x)\}\\
E^{-}:=\{x\in\mathbb{R}^{n}\,\mid\, u_{h}^{-}(x)=u(x-h)+\gamma_{K^{\circ}}(h)<u(x)\}.
\end{array}
\]
The following assertions are easy to check

i) $u_{h}^{\pm}\in W_{u^{-},u^{+}}$.

ii) $E^{\pm}\backslash U$ have measure zero.

iii) $E^{+}=E^{-}-h$.

Now for any $0<\lambda<1$ we have ($i=1,\cdots,m$)
\begin{eqnarray}
 &  & J(u+\lambda(u_{h}^{+}-u))-J(u)=\int_{E^{+}}f(Du(x)+\lambda(Du(x+h)-Du(x)))\nonumber \\
 &  & -\, f(Du(x))+g(u(x)+\lambda(u(x+h)-\gamma_{K^{\circ}}(h)-u(x)))-g(u(x))\, dx\geq0,
\end{eqnarray}
and 
\begin{eqnarray}
 &  & J(u+\lambda(u_{h}^{-}-u))-J(u)=\int_{E^{-}}f(Du(x)+\lambda(Du(x-h)-Du(x)))\nonumber \\
 &  & -\, f(Du(x))+g(u(x)+\lambda(u(x-h)+\gamma_{K^{\circ}}(h)-u(x)))-g(u(x))\, dx\geq0.
\end{eqnarray}
By changing the variable from $x$ to $x+h$ in the last integral,
we get 
\begin{eqnarray}
 &  & \int_{E^{+}}f(Du(x+h)+\lambda(Du(x)-Du(x+h)))-f(Du(x+h))\nonumber \\
 &  & +\, g(u(x+h)+\lambda(u(x)+\gamma_{K^{\circ}}(h)-u(x+h)))-g(u(x+h))\, dx\geq0.
\end{eqnarray}
Adding this to the first integral and using the convexity of $f$,
we have 
\begin{eqnarray}
 &  & \int_{E^{+}}g(u(x+h)+\lambda(u(x)+\gamma_{K^{\circ}}(h)-u(x+h)))-g(u(x+h))\nonumber \\
 &  & +\, g(u(x)+\lambda(u(x+h)-\gamma_{K^{\circ}}(h)-u(x)))-g(u(x))\, dx\geq0.
\end{eqnarray}
We divide this inequality by $\lambda>0$ and take the limit as $\lambda\rightarrow0$.
Then, as $g$ is $C^{1}$ and $u$ is bounded, by Dominated Convergence
Theorem we get 
\begin{equation}
\int_{E^{+}}[g'(u(x+h))-g'(u(x))](u(x)-u(x+h)+\gamma_{K^{\circ}}(h))\, dx\geq0.
\end{equation}
But on $E^{+}$, $u(x)-u(x+h)+\gamma_{K^{\circ}}(h)<0$. Also $g'$
is strictly increasing and therefore $g'(u(x+h))-g'(u(x))>0$. Hence
$E^{+}$ must have measure zero. This means that for a.e. $x\in\mathbb{R}^{n}$
\[
u(x+h)-u(x)\leq\gamma_{K^{\circ}}(h).
\]
Taking $h\rightarrow0$ (through a countable sequence) we get 
\[
D_{h}u(x)\leq\gamma_{K^{\circ}}(h).
\]
Which implies $\gamma_{K}(Du(x))\leq1$, and this is equivalent to
$u\in W_{K}$.

Now suppose that we only have $0\in K$. Let 
\[
K_{i}:=\{x+y\,\mid\, x\in K\,,\,|y|\leq\frac{1}{i}\}=\{z\,\mid\, d(z,K)\leq\frac{1}{i}\}.
\]
Then $\{K_{i}\}$ is a decreasing family of compact convex sets containing
$K$ with $0\in\textrm{int }K_{i}$. Therefore $\{W_{K_{i}}\}$ is
also a decreasing family containing $W_{K}$. Let $u_{i}^{\pm}$ be
the corresponding obstacles to $W_{K_{i}}$. Then we have $u_{i}^{+}\geq u^{+}$
and $u_{i}^{-}\leq u^{-}$. Also $u_{i}^{+}$ decreases with $i$,
and $u_{i}^{-}$ increases with $i$. Thus $\{W_{u_{i}^{-},u_{i}^{+}}\}$
is a decreasing family too and contains $W_{u^{-},u^{+}}$. 

Let $u_{i}$ be the minimizer of $J$ over $W_{K_{i}}$. We have $Du_{0}\in K\subset K_{i}$.
Therefore we can apply the previous argument and we have $J(u_{i})\leq J(v)$
for all $v\in W_{u_{i}^{-},u_{i}^{+}}\supset W_{u^{-},u^{+}}$. Now
as $u_{i}$'s are all in $W_{K_{1}}$ we have $\|u_{i}\|_{W^{1,p}}<C$
for some universal $C$.

Therefore there is a subsequence of $u_{i}$'s, where we denote it
by $u_{i_{k}}$, which converges weakly in $u_{0}+W_{0}^{1,p}(U)$
to $u$. By weak lower semicontinuity of $J$ we get $J(u)\leq\liminf J(u_{i_{k}})\leq J(v)$
for all $v\in W_{u^{-},u^{+}}$. Thus to finish the proof we only
need to show that $u\in W_{K}$. To see this note that the sequence
$u_{i_{k}}$ is eventually in each $W_{K_{i_{k}}}$ and as these are
closed convex sets they are weakly closed, hence $u\in W_{K_{i_{k}}}$
for all $k$. This means $d(Du,K)\leq\frac{1}{i_{k}}$ a.e.. Thus
$d(Du,K)=0$ a.e., and by closedness of $K$ we get the desired result.

Next suppose that $g$ is only convex. Consider the mollifications
$g_{\epsilon}:=\eta_{\epsilon}\star g$, where $\eta_{\epsilon}$
is the standard mollifier. First let us show that $g_{\epsilon}$
is convex too. We have 
\begin{eqnarray*}
g_{\epsilon}(\lambda x+(1-\lambda)y) &  & =\int\eta_{\epsilon}(z)g(\lambda x+(1-\lambda)y-z)\, dz\\
 &  & \leq\int\eta_{\epsilon}(z)[\lambda g(x-z)+(1-\lambda)g(y-z)]\, dz\\
 &  & \leq\lambda g_{\epsilon}(x)+(1-\lambda)g_{\epsilon}(y).
\end{eqnarray*}
Now let 
\[
J_{i}(v):=\int_{U}f(Dv)+g_{\frac{1}{i}}(v)+\frac{1}{i}v^{2}\, dx.
\]
Then since $g_{\epsilon}(v)+\epsilon v^{2}$ is a smooth strictly
convex function, it has strictly increasing derivative. Let $u_{i}$
be the minimizer of $J_{i}$ over $W_{K}$. Then by the above we have
$J_{i}(u_{i})\leq J_{i}(v)$ for all $v\in W_{u^{-},u^{+}}$. As the
$u_{i}$'s are in $W_{K}$, and $W_{K}$ is bounded in $W^{1,p}$,
we can say that there is a subsequence of $u_{i}$, which we continue
to denote it by $u_{i}$, that converges weakly to $u\in W_{K}$. 

Since $g_{\epsilon}$ uniformly converges to $g$ on compact sets,
and for $v\in W_{u^{-},u^{+}}$ we have $\Vert v\Vert_{L^{\infty}}<C$
for some constant $C$ independent of $v$, we have for $\epsilon$
small enough and independent of $v$ 
\begin{equation}
|J_{i}(v)-J(v)|\leq\int_{U}|g_{\frac{1}{i}}(v)-g(v)|+\frac{1}{i}v^{2}\, dx<\delta,
\end{equation}
for $i$ large enough. Hence $J(u_{i})\leq J(v)+2\delta.$ Then by
weak lower semicontinuity of $J$ we have $J(u)\leq\liminf J(u_{i})\leq J(v)+2\delta$.
Since $\delta$ is arbitrary we get that $u$ is the minimizer of
$J$ over $W_{u^{-},u^{+}}$ as required.\end{proof}
\begin{rem}
We can also prove a version of this theorem when $0\notin K$, by
translating $K$. But we need to have a bound on the distance of $K$
and the origin.
\end{rem}

\section{The Equivalence in The Vector-Valued Case}

Suppose $K\subset\mathbb{R}^{n}$ is a balanced compact convex set
whose interior contains $0$. Also suppose that $\boldsymbol{\eta}\in\mathbb{R}^{N}$
is a fixed nonzero vector. Consider the following problems of minimizing
\begin{equation}
I(\mathbf{v}):=\int_{U}|D\mathbf{v}|^{2}-\boldsymbol{\eta}\cdot\mathbf{v}\, dx
\end{equation}
over 
\begin{equation}
K_{1}:=\{\mathbf{v}=(v^{1},\cdots,v^{N})\in H_{0}^{1}(U;\mathbb{R}^{N})\,\mid\,\|D\mathbf{v}\|_{2,K}\leq1\textrm{ a.e.}\},
\end{equation}
and over 
\begin{equation}
K_{2}:=\{\mathbf{v}=(v^{1},\cdots,v^{N})\in H_{0}^{1}(U;\mathbb{R}^{N})\,\mid\,|\mathbf{v}(x)|\leq d_{K}(x,\partial U)\textrm{ a.e.}\}.
\end{equation}
Where 
\begin{equation}
\|A\|_{2,K}:=\underset{z\ne0}{\sup}\,\frac{|Az|}{\gamma_{K}(z)}
\end{equation}
for an $n\times n$ matrix $A$, and $\gamma_{K},d_{K}$ are respectively
the norm associated to $K$ and the metric of that norm. We show that
these problems are equivalent. 

As both $K_{1},K_{2}$ are closed convex sets and $I$ is coercive,
bounded and weakly sequentially lower semicontinuous, both problems
have unique solution. 
\begin{lem}
We have 
\[
K_{1}\subseteq K_{2}.
\]
\end{lem}
\begin{proof}
To see this let $\mathbf{v}\in K_{1}$. Similarly to the proof of
Lemma \ref{Lemma bdd grad} we obtain 
\begin{equation}
|\mathbf{v}(y)-\mathbf{v}(x)|\leq\gamma_{K}(y-x)
\end{equation}
for a.e. $x,y$. Using this relation we can redefine $\mathbf{v}$
on a set of measure zero the same way that we extend Lipschitz functions.
Therefore we can assume that $\mathbf{v}$ is continuous. Now as $\mathbf{v}$
is $0$ on $\partial U$, we can choose $x$ to be the closest point
on $\partial U$ to $y$ with respect to $d_{K}$, and get the desired
result.\end{proof}
\begin{lem}
Let $\mathbf{u}=(u^{1},\cdots,u^{N})$ be the minimizer of $I$ over
$K_{2}$, and let 
\[
T=(T_{l}^{k}):\mathbb{R}^{N}\to\mathbb{R}^{N}
\]
be an orthogonal linear map that fixes $\boldsymbol{\eta}$. Then
$T\mathbf{u}\in K_{2}$ and 
\begin{equation}
I(T\mathbf{u})=I(\mathbf{u}).
\end{equation}
\end{lem}
\begin{proof}
To see this note that $T\mathbf{u}\in H_{0}^{1}(U;\mathbb{R}^{N})$
and as $T$ preserves the norm, for a.e. $x$ we have 
\begin{equation}
|T\mathbf{u}(x)|=|\mathbf{u}(x)|\leq d_{K}(x,\partial U).
\end{equation}
Furthermore as $T$ is orthogonal we have 
\begin{equation}
|DT\mathbf{u}|^{2}=\underset{i}{\sum}\underset{k}{\sum}(T_{l}^{k}D_{i}u^{l})^{2}=\underset{i}{\sum}\underset{l}{\sum}(D_{i}u^{l})^{2}=|D\mathbf{u}|^{2}.
\end{equation}
Hence (since $T\boldsymbol{\eta}=\boldsymbol{\eta}$ and $T$ is orthogonal)
\begin{eqnarray*}
 & I(T\mathbf{u}) & =\int_{U}|DT\mathbf{u}|^{2}-\boldsymbol{\eta}\cdot T\mathbf{u}\, dx\\
 &  & =\int_{U}|D\mathbf{u}|^{2}-T\boldsymbol{\eta}\cdot T\mathbf{u}\, dx\\
 &  & =\int_{U}|D\mathbf{u}|^{2}-\boldsymbol{\eta}\cdot\mathbf{u}\, dx=I(\mathbf{u}).
\end{eqnarray*}
\end{proof}
\begin{thm}
\label{vector-to-scalar}We have 
\begin{equation}
\mathbf{u}(x)=u(x)\boldsymbol{\eta},
\end{equation}
where $u$ is the minimizer of 
\begin{equation}
J_{1}(v):=\int_{U}|Dv|^{2}-v\, dx
\end{equation}
over 
\begin{equation}
K_{3}:=\{v\in H_{0}^{1}(U;\mathbb{R})\,\mid\,|v(x)|\leq\frac{1}{|\boldsymbol{\eta}|}d_{K}(x,\partial U)\textrm{ a.e.}\}.
\end{equation}
\end{thm}
\begin{proof}
By the above lemma and uniqueness of the minimizer, we must have $T\mathbf{u}=\mathbf{u}$
for all orthogonal linear maps $T$ that fix $\boldsymbol{\eta}$.
This implies that $\mathbf{u}(x)=u(x)\boldsymbol{\eta}$ for some
scalar function $u$. Now we have 
\[
|u(x)\boldsymbol{\eta}|=|\mathbf{u}|\leq d_{K}(x,\partial U).
\]
Hence for a.e. $x$ 
\begin{equation}
|u(x)|\leq\frac{1}{|\boldsymbol{\eta}|}d_{K}(x,\partial U).
\end{equation}

Also we have 
\[
D_{i}\mathbf{u}=D_{i}u\boldsymbol{\eta}.
\]
Thus 
\begin{equation}
I(\mathbf{u})=\int_{U}|\boldsymbol{\eta}|^{2}|Du|^{2}-|\boldsymbol{\eta}|^{2}u\, dx=|\boldsymbol{\eta}|^{2}\int_{U}|Du|^{2}-u\, dx=|\boldsymbol{\eta}|^{2}J_{1}(u).
\end{equation}
It is easy to see that $u$ is the minimizer of $J_{1}$ over $K_{3}$.
Because for any $w\in K_{3}$ we have $w\boldsymbol{\eta}\in K_{2}$,
therefore 
\[
J_{1}(u)=|\boldsymbol{\eta}|^{-2}I(u\boldsymbol{\eta})=|\boldsymbol{\eta}|^{-2}I(\mathbf{u})\leq|\boldsymbol{\eta}|^{-2}I(w\boldsymbol{\eta})=J_{1}(w).
\]
\end{proof}
\begin{thm}
The minimizer of $I$ over $K_{2}$ is the same as its minimizer over
$K_{1}$.\end{thm}
\begin{proof}
By the above theorem 
\[
\mathbf{u}(x)=u(x)\boldsymbol{\eta},
\]
where $u$ is the minimizer of $J_{1}$ over $K_{3}$. But we know
that the minimizer of $J_{1}$ over $K_{3}$ is the same as its minimizer
over
\begin{equation}
K_{4}:=\{v\in H_{0}^{1}(U;\mathbb{R})\,\mid\,\gamma_{K^{\circ}}(Dv)\leq\frac{1}{|\boldsymbol{\eta}|}\textrm{ a.e.}\}.
\end{equation}
Therefore for all $z\in\mathbb{R}^{n}$, we have a.e. 
\begin{eqnarray*}
 & |D\mathbf{u}\cdot z|^{2} & =\underset{l}{\sum}\underset{i}{\sum}(D_{i}u^{l}z^{i})^{2}=\underset{l}{\sum}\underset{i}{\sum}(D_{i}u\eta^{l}z^{i})^{2}\\
 &  & =\underset{i}{\sum}(D_{i}uz^{i})^{2}\underset{l}{\sum}(\eta^{l})^{2}=|\boldsymbol{\eta}|^{2}|Du\cdot z|^{2}\\
 &  & \le|\boldsymbol{\eta}|^{2}\gamma_{K^{\circ}}(Du)^{2}\gamma_{K}(z)^{2}\leq\gamma_{K}(z)^{2}.
\end{eqnarray*}
This means that 
\begin{equation}
\|D\mathbf{u}\|_{2,K}\leq1\qquad\textrm{ a.e.}.
\end{equation}
Hence $\mathbf{u}\in K_{1}$. Since $K_{1}\subseteq K_{2}$, $\mathbf{u}$
is also the minimizer of $I$ over $K_{1}$.
\end{proof}

\section{The Optimal Regularity}

Let 
\[
J_{\eta}(v):=\int_{U}\frac{1}{2}|Dv|^{2}-\eta v\, dx.
\]
Suppose $K\subset\mathbb{R}^{n}$ is a balanced compact convex set
whose interior contains $0$. Let $u$ be the minimizer of $J_{\eta}$
over 
\[
W_{K}:=\{v\in c+H_{0}^{1}(U)\,\mid\,\gamma_{K}(Dv)\leq k\textrm{ a.e.}\},
\]
where $c,k$ are constants and $\gamma_{K}$ is the gauge function
of $K$. We showed that $u$ is also the minimizer of $J_{\eta}$
over 
\[
\{v\in c+H_{0}^{1}(U)\,\mid\, c-kd_{K^{\circ}}(x,\partial U)\leq v(x)\leq c+kd_{K^{\circ}}(x,\partial U)\textrm{ a.e.}\},
\]
where $K^{\circ}$ is the polar of $K$, and $d_{K^{\circ}}$ is the
metric associated to the norm $\gamma_{K^{\circ}}$.

By the above assumptions, there is $A>0$ such that $\gamma_{K^{\circ}}(x)\leq A|x|$
for all $x$. We also need some sort of bound on the second derivative
of $\gamma_{K^{\circ}}$, hence we assume that 
\begin{equation}
\frac{\gamma_{K^{\circ}}(x+hz)+\gamma_{K^{\circ}}(x-hz)-2\gamma_{K^{\circ}}(x)}{h^{2}}\leq\frac{B}{\gamma_{K^{\circ}}(x)-h},
\end{equation}
where $B\geq1$ is a constant, $\gamma_{K^{\circ}}(z)=1$ and $h<\gamma_{K^{\circ}}(x)$. 
\begin{lem}
The above inequality holds when $\gamma_{K^{\circ}}$ is the $p\hspace{1bp}$-norm
for $p\geq2$. (In this case, $K$ is the unit disk in the $\frac{p}{p-1}$-norm.)\end{lem}
\begin{proof}
Let $\gamma_{p}(x)=(\sum|x_{i}|^{p})^{1/p}$ then for $\gamma_{p}(x)\neq0$
we have 
\begin{equation}
D_{i}\gamma_{p}(x)=|x_{i}|^{p-1}\textrm{sgn}(x_{i})(\sum|x_{j}|^{p})^{1/p-1}=\frac{|x_{i}|^{p-1}\textrm{sgn}(x_{i})}{\gamma_{p}(x)^{p-1}},
\end{equation}
where $\textrm{sgn}(x_{i})$ is the sign of $x_{i}$. Thus 
\begin{equation}
D_{ij}^{2}\gamma_{p}(x)=(p-1)|x_{i}|^{p-2}\delta_{ij}\frac{1}{\gamma_{p}(x)^{p-1}}-(p-1)|x_{i}|^{p-1}|x_{j}|^{p-1}\frac{\textrm{sgn}(x_{i})\textrm{sgn}(x_{j})}{\gamma_{p}(x)^{2p-1}}.
\end{equation}
Hence 
\begin{eqnarray*}
 & D_{zz}^{2}\gamma_{p}(x) & =\sum D_{ij}^{2}\gamma_{p}(x)z_{i}z_{j}\\
 &  & =\frac{p-1}{\gamma_{p}(x)^{p-1}}\sum|x_{i}|^{p-2}z_{i}^{2}-\frac{p-1}{\gamma_{p}(x)^{2p-1}}(\sum\textrm{sgn}(x_{i})|x_{i}|^{p-1}z_{i})^{2}.
\end{eqnarray*}
By Holder's inequality we get 
\[
D_{zz}^{2}\gamma_{p}(x)\leq\frac{p-1}{\gamma_{p}(x)^{p-1}}(\sum(|x_{i}|^{p-2})^{\frac{p}{p-2}})^{\frac{p-2}{p}}(\sum(z_{i}^{2})^{\frac{p}{2}})^{\frac{2}{p}}=\frac{p-1}{\gamma_{p}(x)}\gamma_{p}(z)^{2}.
\]
Thus if $\gamma_{p}(z)=1$, we have 
\begin{equation}
D_{zz}^{2}\gamma_{p}(x)\leq\frac{p-1}{\gamma_{p}(x)}.
\end{equation}
When $\gamma_{p}(x)>h$, $\gamma_{p}$ is nonzero on the segment $L:=\{x+\tau z\,\mid\,-h\leq\tau\leq h\}$;
and so it is twice differentiable there. Therefore we can apply the
mean value theorem to the restriction of $\gamma_{p}$ and its first
derivative to the segment $L$. Hence we get 
\begin{eqnarray*}
\frac{\gamma_{p}(x+hz)+\gamma_{p}(x-hz)-2\gamma_{p}(x)}{h^{2}} &  & =\frac{\gamma_{p}(x+hz)-\gamma_{p}(x)+\gamma_{p}(x-hz)-\gamma_{p}(x)}{h^{2}}\\
 &  & =\frac{hD_{z}\gamma_{p}(x+sz)-hD_{z}\gamma_{p}(x-tz)}{h^{2}}\\
 &  & =\frac{(s+t)}{h}D_{zz}^{2}\gamma_{p}(x+rz),
\end{eqnarray*}
where $0<s,t<h$ and $-t<r<s$. Now as $\gamma_{p}$ is convex, its
second derivative is nonnegative definite. Hence 
\begin{eqnarray}
\frac{\gamma_{p}(x+hz)+\gamma_{p}(x-hz)-2\gamma_{p}(x)}{h^{2}} &  & \leq2D_{zz}^{2}\gamma_{p}(x+rz)\nonumber \\
 &  & \leq\frac{2(p-1)}{\gamma_{p}(x+rz)}\nonumber \\
 &  & \leq\frac{2(p-1)}{\gamma_{p}(x)-h}.
\end{eqnarray}
In the last inequality we used the triangle inequality for $\gamma_{p}$.
\end{proof}
The following is our main regularity result. Note that by Theorem
\ref{vector-to-scalar}, we also get the regularity for the vector-valued
case.
\begin{thm}
Suppose $u$ is the minimizer of $J_{\eta}$ over $W_{K}$. Then $u\in W_{\textrm{loc}}^{2,\infty}(U)$,
and 
\begin{equation}
|D^{2}u(x)|\leq C(n)\big[|\eta|+\frac{kA^{2}B}{d_{K^{\circ}}(x,\partial U)}+\frac{A^{2}|c|}{(d_{K^{\circ}}(x,\partial U))^{2}}\big],\label{eq: 2nd der bnd}
\end{equation}
where $C(n)$ is a constant depending only on the dimension $n$.\end{thm}
\begin{proof}
Let us assume that $U$ has smooth boundary, we will remove this restriction
at the end. We know that 
\[
\phi(x)=c-kd_{K^{\circ}}(x,\partial U)\leq u(x)\leq c+kd_{K^{\circ}}(x,\partial U)=\psi(x).
\]

Let $\phi_{\epsilon}=\eta_{\epsilon}\star\phi+\delta_{\epsilon}$
and $\psi_{\epsilon}=\eta_{\epsilon}\star\psi$ where $\eta_{\epsilon}$
is the standard mollifier and $4kA\epsilon<\delta_{\epsilon}<5kA\epsilon$
is chosen such that $\partial\{\phi_{\epsilon}<\psi_{\epsilon}\}$
is $C^{\infty}$ (which is possible by Sard's Theorem). Note that
\begin{eqnarray}
\{x\in U\,\mid\, d_{K^{\circ}}(x,\partial U)>4A\epsilon\} &  & \subset\{x\in\bar{U}\,\mid\,\phi_{\epsilon}(x)\leq\psi_{\epsilon}(x)\}\nonumber \\
 &  & \subset\{x\in U\,\mid\, d_{K^{\circ}}(x,\partial U)>A\epsilon\},
\end{eqnarray}
as $\psi(x)-\phi(x)=2kd_{K^{\circ}}(x,\partial U)$. Also 
\begin{eqnarray*}
 & |\psi_{\epsilon}(x)-\psi(x)| & \leq\int_{|y|\leq\epsilon}\eta_{\epsilon}(y)|\psi(x-y)-\psi(x)|\, dy\\
 &  & \le\int_{|y|\leq\epsilon}\eta_{\epsilon}(y)k\gamma_{K^{\circ}}(y)\, dy\\
 &  & \leq kA\epsilon\int_{|y|\leq\epsilon}\eta_{\epsilon}(y)\, dy\;=kA\epsilon.
\end{eqnarray*}
Similarly $|\eta_{\epsilon}\star\phi-\phi|\leq kA\epsilon$. 

We can easily show that $\gamma_{K}(D\phi_{\epsilon})\leq k$ and
$\gamma_{K}(D\psi_{\epsilon})\leq k$. Because of Jensen's inequality
and convexity of $\gamma_{K}$, we have 
\begin{eqnarray*}
 & \gamma_{K}(D\phi_{\epsilon}(x)) & \leq\int\gamma_{K}(\eta_{\epsilon}(y)D\phi(x-y))\, dy\\
 &  & =\int\eta_{\epsilon}(y)\gamma_{K}(D\phi(x-y))\, dy\\
 &  & \leq k\int\eta_{\epsilon}(y)\, dy\;=k.
\end{eqnarray*}

Let $U_{\epsilon}:=\{x\in U\,\mid\,\phi_{\epsilon}(x)<\psi_{\epsilon}(x)\}$,
and denote by $u_{\epsilon}$ the minimizer of $J_{\eta}$ over $\{v\in H^{1}(D_{\epsilon})\,\mid\,\phi_{\epsilon}\leq v\leq\psi_{\epsilon}\textrm{ a.e. }\}$.
Set 
\begin{equation}
\begin{array}{c}
N_{\epsilon}:=\{x\in U_{\epsilon}\,\mid\,\phi_{\epsilon}(x)<u_{\epsilon}(x)<\psi_{\epsilon}(x)\}\\
\Lambda_{1}:=\{x\in U_{\epsilon}\,\mid\, u_{\epsilon}(x)=\phi_{\epsilon}(x)\}\\
\Lambda_{2}:=\{x\in U_{\epsilon}\,\mid\, u_{\epsilon}(x)=\psi_{\epsilon}(x)\}.
\end{array}
\end{equation}
Since $\phi_{\epsilon},\psi_{\epsilon}$ are smooth, $u_{\epsilon}\in W^{2,p}(U_{\epsilon})$
for any $1<p<\infty$. Therefore $N_{\epsilon}$ is open and $\Lambda_{i}$'s
are closed. Also we define the free boundaries $F_{i}:=\partial\Lambda_{i}\cap U_{\epsilon}$.
Note that $\partial N_{\epsilon}$ consists of $F_{i}$'s and part
of $\partial U_{\epsilon}$.

Our strategy for the proof is to show that $u_{\epsilon}$ satisfies
the bound (\ref{eq: 2nd der bnd}) on $U_{\epsilon}$. Then we can
let $\epsilon\rightarrow0$. Since $\phi_{\epsilon}\rightarrow\phi$
, $\psi_{\epsilon}\rightarrow\psi$ uniformly, we have $u_{\epsilon}\rightarrow u$
uniformly. Also as for small enough $\epsilon$, $u_{\epsilon}$'s
are bounded in $W^{2,\infty}(V)$ for $V\subset\subset U$, a subsequence
of them is weakly star convergent, and the limit is $u$. Therefore
$u\in W_{\textrm{loc}}^{2,\infty}(U)$ and 
\[
|D^{2}u|_{L^{\infty}}\leq\liminf|D^{2}u_{\epsilon}|_{L^{\infty}}
\]
gives the desired bound.

Now suppose $\partial U$ is not smooth. We approximate $U$ by a
shrinking sequence $U_{i}$ of larger domains with smooth boundaries.
Let $u_{i}$ be the minimizer of $J_{\eta}$ on $U_{i}$, then $u_{i}\to u$
uniformly. To see this note that we can consider $u$ as a function
on $U_{i}$, thus $J_{\eta}(u_{i})\le J_{\eta}(u)$. An argument similar
to the above implies that a subsequence of $u_{i}$'s converges weakly
star to a function $u^{\star}$, and $u^{\star}$ satisfies the desired
bound. But $u^{\star}\in W_{K}$. Also the lower semicontinuity of
$J_{\eta}$ implies that $J_{\eta}(u^{\star})\le J_{\eta}(u)$. Since
the other inequality is satisfied too, we have $J_{\eta}(u^{\star})=J_{\eta}(u)$.
The uniqueness of the minimizer implies that $u^{\star}=u$. Hence
$u$ satisfies the bound (\ref{eq: 2nd der bnd}) too.
\end{proof}
Now let us start proving the bound (\ref{eq: 2nd der bnd}) for $u_{\epsilon}$.
\begin{lem}
We have 
\[
\gamma_{K}(Du_{\epsilon})\leq k
\]
on $U_{\epsilon}$. \end{lem}
\begin{proof}
Since on $\partial U_{\epsilon}$ we have $u_{\epsilon}=\phi_{\epsilon}=\psi_{\epsilon}$
we get $D_{z}u_{\epsilon}=D_{z}\phi_{\epsilon}=D_{z}\psi_{\epsilon}$
for any direction $z$ tangent to $\partial U_{\epsilon}$, and as
$u_{\epsilon}$ is between the obstacles inside $U_{\epsilon}$ we
have $D_{\nu}\phi_{\epsilon}\leq D_{\nu}u_{\epsilon}\leq D_{\nu}\psi_{\epsilon}$
where $\nu$ is the normal direction to $\partial U_{\epsilon}$.
Therefore $Du_{\epsilon}$ is a convex combination of $D\phi_{\epsilon},D\psi_{\epsilon}$
and we get the bound on $\partial U_{\epsilon}$ by convexity of $\gamma_{K}$.
The bound holds on $\Lambda_{i}$'s (and hence on $F_{i}$'s) obviously
as $u_{\epsilon}$ equals one of the obstacles there. 

To obtain the bound for $N_{\epsilon}$ note that for any vector $z$
with $\gamma_{K^{\circ}}(z)=1$ we have 
\[
|D_{z}u_{\epsilon}|=|z\cdot Du_{\epsilon}|\leq\gamma_{K^{\circ}}(z)\gamma_{K}(Du_{\epsilon})\leq k
\]
on $\partial N_{\epsilon}$, and as $D_{z}u_{\epsilon}$ is harmonic
in $N_{\epsilon}$ we get $|D_{z}u_{\epsilon}|\leq k$ in $N_{\epsilon}$
by maximum principle. The result follows from $\gamma_{K}(Du_{\epsilon})=\underset{\gamma_{K^{\circ}}(z)=1}{\sup}|D_{z}u_{\epsilon}|$.
\end{proof}
The local behavior of the free boundaries is the same as the case
of one obstacle problem as obstacles do not touch inside $U_{\epsilon}$.
We need the following lemma from \citet{MR679313}.
\begin{lem}
The free boundary has measure zero. Furthermore for any direction
$z$

${\rm (i)}$ if $y\in N_{\epsilon}$ approaches $x\in F_{1}$, then
$\underset{y\rightarrow x}{\liminf}D_{zz}^{2}(u_{\epsilon}-\phi_{\epsilon})(y)\geq0$.

${\rm (ii)}$ If $y\in N_{\epsilon}$ approaches $x\in F_{2}$, then
$\underset{y\rightarrow x}{\liminf}D_{zz}^{2}(\psi_{\epsilon}-u_{\epsilon})(y)\geq0$.
\end{lem}

\begin{lem}
For any direction $z$ with $|z|=1$, we have 
\begin{equation}
\begin{array}{c}
D_{zz}^{2}\phi_{\epsilon}(x)\geq-\frac{kA^{2}B}{d_{K^{\circ}}(x,\partial U)-A\epsilon}\\
\\
D_{zz}^{2}\psi_{\epsilon}(x)\leq\frac{kA^{2}B}{d_{K^{\circ}}(x,\partial U)-A\epsilon}
\end{array}
\end{equation}
for all $x\in U$ with $d_{K^{\circ}}(x,\partial U)>A\epsilon$.\end{lem}
\begin{proof}
First we assume $\gamma_{K^{\circ}}(z)=1$. Let $x_{0}\in U$ then
\[
\psi(x_{0})=c+kd_{K^{\circ}}(x_{0},\partial U)=c+k\gamma_{K^{\circ}}(x_{0}-y_{0})
\]
for some $y_{0}\in\partial U$. Set $\gamma(x)=c+k\gamma_{K^{\circ}}(x-y_{0})$.
Then $\psi(x)\leq\gamma(x)$ and $\psi(x_{0})=\gamma(x_{0})$. Now
for $h<\gamma_{K^{\circ}}(x_{0}-y_{0})$ we have 
\begin{equation}
\Delta_{h,z}^{2}\psi(x_{0})=\frac{\psi(x_{0}+hz)+\psi(x_{0}-hz)-2\psi(x_{0})}{h^{2}}\leq\Delta_{h,z}^{2}\gamma(x_{0}).
\end{equation}
By our assumption 
\[
\Delta_{h,z}^{2}\gamma(x_{0})\leq\frac{kB}{\gamma_{K^{\circ}}(x_{0}-y_{0})-h}=\frac{kB}{d_{K^{\circ}}(x_{0},\partial U)-h}.
\]
Hence $\Delta_{h,z}^{2}\psi(x)\leq\frac{kB}{d_{K^{\circ}}(x,\partial U)-h}$
for $d_{K^{\circ}}(x,\partial U)>h$. 

Now for $d_{K^{\circ}}(x,\partial U)>h+A\epsilon$, we have 
\begin{eqnarray*}
 & \Delta_{h,z}^{2}\psi_{\epsilon}(x) & =\int_{|y|<\epsilon}\eta_{\epsilon}(y)\Delta_{h,z}^{2}\psi(x-y)\, dy\\
 &  & \leq\int_{|y|<\epsilon}\eta_{\epsilon}(y)\frac{kB}{d_{K^{\circ}}(x-y,\partial U)-h}\, dy\\
 &  & \leq\int_{|y|<\epsilon}\eta_{\epsilon}(y)\frac{kB}{d_{K^{\circ}}(x,\partial U)-A\epsilon-h}\, dy\\
 &  & =\frac{kB}{d_{K^{\circ}}(x,\partial U)-A\epsilon-h}.
\end{eqnarray*}
Here we used the fact that 
\begin{eqnarray*}
 & d_{K^{\circ}}(x-y,\partial U) & \geq d_{K^{\circ}}(x,\partial U)-\gamma_{K^{\circ}}(y)\\
 &  & \ge d_{K^{\circ}}(x,\partial U)-A|y|\\
 &  & >d_{K^{\circ}}(x,\partial U)-A\epsilon\;>h.
\end{eqnarray*}
Taking $h\rightarrow0$, we get for $d_{K^{\circ}}(x,\partial U)>A\epsilon$
\[
D_{zz}^{2}\psi_{\epsilon}(x)\leq\frac{kB}{d_{K^{\circ}}(x,\partial U)-A\epsilon}.
\]
Now if we take $|z|=1$ and apply the above result to $w=\frac{z}{\gamma_{K^{\circ}}(z)}$,
we get 
\[
D_{zz}^{2}\psi_{\epsilon}(x)=(\gamma_{K^{\circ}}(z))^{2}D_{ww}^{2}\psi_{\epsilon}(x)\leq A^{2}D_{ww}^{2}\psi_{\epsilon}(x)\leq\frac{kA^{2}B}{d_{K^{\circ}}(x,\partial U)-A\epsilon},
\]
as $\gamma_{K^{\circ}}(z)\leq A$ and $D^{2}\psi_{\epsilon}$ is nonnegative
since $\psi$ is convex. The inequality for $\phi_{\epsilon}$ follows
from $D^{2}\phi_{\epsilon}=-D^{2}\psi_{\epsilon}$.\end{proof}
\begin{lem}
For any direction $z$ with $|z|=1$ 
\begin{equation}
|D_{zz}^{2}u_{\epsilon}(x)|\leq C(n)\big[|\eta|+\frac{kA^{2}B}{d_{K^{\circ}}(x,\partial U)-A\epsilon}+\frac{A^{2}|c|}{(d_{K^{\circ}}(x,\partial U)-A\epsilon)^{2}}\big]
\end{equation}
for a.e. $x\in U_{\epsilon}$, where $C(n)$ is a constant depending
only on the dimension $n$. \end{lem}
\begin{proof}
Since $u_{\epsilon}\in W^{2,p}(U_{\epsilon})$ we have $D_{zz}^{2}u_{\epsilon}=D_{zz}^{2}\phi_{\epsilon}$
a.e. on $\Lambda_{1}$. Also in a $U_{\epsilon}-$neighborhood of
$\Lambda_{1}$ we have $-\Delta u_{\epsilon}\geq\eta$ a.e., since
$u_{\epsilon}$ solves the variational inequality there. Thus for
a.e. $x\in\Lambda_{1}$ 
\begin{eqnarray}
-\frac{kA^{2}B}{d_{K^{\circ}}(x,\partial U)-A\epsilon} &  & \leq D_{zz}^{2}\phi_{\epsilon}(x)\nonumber \\
 &  & =D_{zz}^{2}u_{\epsilon}(x)\nonumber \\
 &  & =\Delta u_{\epsilon}(x)-\sum D_{z_{i}z_{i}}^{2}u_{\epsilon}(x)\nonumber \\
 &  & \leq-\eta-\sum D_{z_{i}z_{i}}^{2}\phi_{\epsilon}(x)\nonumber \\
 &  & \leq|\eta|+\frac{(n-1)kA^{2}B}{d_{K^{\circ}}(x,\partial U)-A\epsilon},
\end{eqnarray}
where $\{z,z_{i}\}$ form an orthonormal system (Note that for $x\in U_{\epsilon}$
we have $d_{K^{\circ}}(x,\partial U)>A\epsilon$). Similarly, using
$\psi_{\epsilon}$ we obtain that for a.e. $x\in\Lambda_{2}$ 
\begin{equation}
-|\eta|-\frac{(n-1)kA^{2}B}{d_{K^{\circ}}(x,\partial U)-A\epsilon}\leq D_{zz}^{2}u_{\epsilon}(x)\leq\frac{kA^{2}B}{d_{K^{\circ}}(x,\partial U)-A\epsilon}.
\end{equation}

It only remains to obtain the bound on $N_{\epsilon}$. We do this
using maximum principle, since $D_{zz}^{2}u_{\epsilon}$ is harmonic
in $N_{\epsilon}$. Therefore we need to estimate $D_{zz}^{2}u_{\epsilon}$
near $F_{i}$ and $\partial U_{\epsilon}$. First, for $x\in F_{1}$
and $y\in N_{\epsilon}$, we have by continuity of $D^{2}\phi_{\epsilon}$
\begin{equation}
\underset{y\rightarrow x}{\liminf}D_{zz}^{2}u_{\epsilon}(y)\geq\underset{y\rightarrow x}{\liminf}D_{zz}^{2}\phi_{\epsilon}(y)=D_{zz}^{2}\phi_{\epsilon}(x)\geq-\frac{kA^{2}B}{d_{K^{\circ}}(x,\partial U)-A\epsilon}.
\end{equation}
This is true for the $z_{i}$ directions too. Also 
\begin{equation}
\underset{y\rightarrow x}{\limsup}(D_{zz}^{2}u_{\epsilon}(y)+\sum D_{z_{i}z_{i}}^{2}u_{\epsilon}(y))=\underset{y\rightarrow x}{\limsup}\Delta u_{\epsilon}(y)=-\eta.
\end{equation}
Thus 
\begin{eqnarray}
\underset{y\rightarrow x}{\limsup}D_{zz}^{2}u_{\epsilon}(y) &  & \le\underset{y\rightarrow x}{\limsup}(D_{zz}^{2}u_{\epsilon}(y)+\sum D_{z_{i}z_{i}}^{2}u_{\epsilon}(y))-\sum\underset{y\rightarrow x}{\liminf}D_{z_{i}z_{i}}^{2}u_{\epsilon}(y)\nonumber \\
 &  & =-\eta-\sum\underset{y\rightarrow x}{\liminf}D_{z_{i}z_{i}}^{2}u_{\epsilon}(y)\\
 &  & \leq|\eta|+\frac{(n-1)kA^{2}B}{d_{K^{\circ}}(x,\partial U)-A\epsilon}.
\end{eqnarray}
Similarly on $F_{2}$ we have 
\begin{eqnarray}
-|\eta|-\frac{(n-1)kA^{2}B}{d_{K^{\circ}}(x,\partial U)-A\epsilon} &  & \leq\underset{y\rightarrow x}{\liminf}D_{zz}^{2}u_{\epsilon}(y)\nonumber \\
\le\underset{y\rightarrow x}{\limsup}D_{zz}^{2}u_{\epsilon}(y) &  & \leq\frac{kA^{2}B}{d_{K^{\circ}}(x,\partial U)-A\epsilon}.
\end{eqnarray}

Next we show that 
\begin{equation}
|D_{zz}^{2}u_{\epsilon}(x)|\leq C(n)[|\eta|+\frac{kAB}{r}+\frac{|c|}{r^{2}}]\qquad x\in N_{\epsilon}\; d_{K^{\circ}}(x,\partial U_{\epsilon})=Ar,
\end{equation}
for fixed and small $r$ and $\epsilon<r/16$. Note that 
\[
d_{K^{\circ}}(B_{r/2}(x),\partial U_{\epsilon})>Ar-Ar/2=Ar/2.
\]

Fix $x_{0}\in N_{\epsilon}$ with $d_{K^{\circ}}(x_{0},\partial U_{\epsilon})=Ar$
and consider the function $v_{\epsilon}(y)=u_{\epsilon}(x_{0}+ry)$
in $B_{1}(0)$. Then by known bounds on $u_{\epsilon}$ we have in
$B_{1/2}(0)$ 
\begin{equation}
\begin{array}{c}
|v_{\epsilon}|\leq|c|+6Ak\epsilon+3Akr/2<|c|+2Akr\\
\\
\gamma_{K}(Dv_{\epsilon})\leq rk.
\end{array}
\end{equation}
Also for a.e. $y\in B_{1/2}(0)$ we have 
\begin{eqnarray}
 & |\Delta v_{\epsilon}(y)| & \leq nr^{2}(|\eta|+\frac{(n-1)kA^{2}B}{d_{K^{\circ}}(x_{0}+ry,\partial U)-A\epsilon})\nonumber \\
 &  & \le n(|\eta|+\frac{(n-1)kA^{2}B}{Ar/2-A\epsilon})r^{2}\nonumber \\
 &  & <n(|\eta|+\frac{16(n-1)kAB}{7r})r^{2}.
\end{eqnarray}
Since $\Delta u_{\epsilon}=-\eta$ in $N_{\epsilon}$ and it is bounded
on $\Lambda_{i}$'s (and free boundaries have measure zero). Choose
$\sigma\in C_{0}^{\infty}(B_{1/2}(0))$ such that $\sigma=1$ in $B_{1/4}(0)$.
Then in $B_{1/2}(0)$ 
\begin{eqnarray}
 & |\Delta(\sigma v_{\epsilon})| & =|(\Delta\sigma)v_{\epsilon}+2D\sigma\cdot Dv_{\epsilon}+\sigma\Delta v_{\epsilon}|\nonumber \\
 &  & \le C(n)[(|\eta|+\frac{kAB}{r})r^{2}+|c|].
\end{eqnarray}
By elliptic theory it follows 
\begin{equation}
|\sigma v_{\epsilon}|_{W^{2,p}(B_{1/2}(0))}\leq C(n,p)[(|\eta|+\frac{kAB}{r})r^{2}+|c|],
\end{equation}
for any $1<p<\infty$ (note that boundary term is zero). In particular
\begin{equation}
|D_{ij}^{2}v_{\epsilon}|_{L^{p}(B_{1/4}(0))}\leq C(n,p)[(|\eta|+\frac{kAB}{r})r^{2}+|c|].
\end{equation}
We want to extend this to $p=\infty$.

Let $\tau\in C_{0}^{\infty}(B_{1/4}(0))$ with $\tau=1$ in $B_{1/8}(0)$.
Consider the open set $N=\{y\,\mid\, x_{0}+ry\in N_{\epsilon}\}$.
In $N$ we have $\Delta v_{\epsilon}=-\eta r^{2}$. Thus (note that
$v_{\epsilon}$ is smooth in $N$) 
\begin{equation}
\Delta D_{zz}^{2}(\tau v_{\epsilon})=D_{z}h,
\end{equation}
where 
\begin{eqnarray}
h & :=D_{z}\Delta(\tau v_{\epsilon}) & =D_{z}((\Delta\sigma)v_{\epsilon}+2D\sigma\cdot Dv_{\epsilon}+\sigma\Delta v_{\epsilon})\nonumber \\
 &  & =D_{z}((\Delta\sigma)v_{\epsilon}+2D\sigma\cdot Dv_{\epsilon}-\sigma\eta r^{2}).
\end{eqnarray}
Using the above estimates we find that 
\begin{equation}
|h|_{L^{p}(N)}\leq C(n,p)[(|\eta|+\frac{kAB}{r})r^{2}+|c|].
\end{equation}
Now take 
\begin{equation}
V(y)=\begin{cases}
\alpha_{n}|y|^{2-n} & n\geq3\\
\\
\alpha_{2}\log|y| & n=2
\end{cases}
\end{equation}
to be the fundamental solution of $-\Delta$. Then 
\[
g(y)=-\int_{N}\frac{\partial V(y-w)}{\partial z}h(w)\, dw
\]
satisfies 
\[
\Delta g=\frac{\partial h}{\partial z}
\]
in $N$. By the bound on $h$ we find that 
\begin{equation}
|g|_{L^{\infty}(N)}\leq C(n)[(|\eta|+\frac{kAB}{r})r^{2}+|c|],
\end{equation}
since for $p>n$, $\frac{\partial V}{\partial z}$ is in $L^{q}$
where $q$ is the dual exponent of $p$. The function $D_{zz}^{2}(\tau v_{\epsilon})-g$
is then harmonic in $N\cap B_{1/4}(0)$. The boundary of this set
consists of part of $\partial B_{1/4}(0)$ in which $\tau=0$ and
$g$ is bounded, and another part inside $B_{1/4}(0)$ where corresponds
to the free boundaries and both $g,\, D_{zz}^{2}(\tau v_{\epsilon})$
are bounded there by the above bounds. Therefore by the maximum principle
we get 
\begin{equation}
|D_{zz}^{2}v_{\epsilon}(0)|\leq C(n)[(|\eta|+\frac{kAB}{r})r^{2}+|c|].
\end{equation}
Hence 
\begin{eqnarray*}
 & |D_{zz}^{2}u_{\epsilon}(x_{0})| & \leq C(n)[|\eta|+\frac{kAB}{r}+\frac{|c|}{r^{2}}]\\
 &  & =C(n)[|\eta|+\frac{kA^{2}B}{d_{K^{\circ}}(x_{0},\partial U_{\epsilon})}+\frac{A^{2}|c|}{(d_{K^{\circ}}(x_{0},\partial U_{\epsilon}))^{2}}].
\end{eqnarray*}
The proof of the lemma is complete once we notice that for $x\in U_{\epsilon}$
\[
d_{K^{\circ}}(x,\partial U_{\epsilon})\geq d_{K^{\circ}}(x,\partial U)-A\epsilon.
\]
\end{proof}
\begin{acknowledgement*}
I would like to express my gratitude to Lawrence C. Evans and Nicolai
Reshetikhin for their invaluable help with this research.
\end{acknowledgement*}

\bibliographystyle{abbrvnat}
\setcitestyle{numerical,open={((},close={))}}
\bibliography{New-Paper-Bibliography}

\def\cprime{$'$} \def\cprime{$'$} \def\cprime{$'$} \def\cprime{$'$}
  \def\cprime{$'$} \def\cprime{$'$} \def\cprime{$'$}
\begin{thebibliography}{14}
\providecommand{\natexlab}[1]{#1}
\providecommand{\url}[1]{\texttt{#1}}
\expandafter\ifx\csname urlstyle\endcsname\relax
  \providecommand{\doi}[1]{doi: #1}\else
  \providecommand{\doi}{doi: \begingroup \urlstyle{rm}\Url}\fi

\bibitem[Brezis and Sibony(1971)]{MR0346345}
H.~Brezis and M.~Sibony.
\newblock \'{E}quivalence de deux in\'equations variationnelles et
  applications.
\newblock \emph{Arch. Rational Mech. Anal.}, 41:\penalty0 254--265, 1971.

\bibitem[Brezis and Stampacchia(1968)]{MR0239302}
H.~Brezis and G.~Stampacchia.
\newblock Sur la r\'egularit\'e de la solution d'in\'equations elliptiques.
\newblock \emph{Bull. Soc. Math. France}, 96:\penalty0 153--180, 1968.

\bibitem[Caffarelli and Rivi{\`e}re(1979)]{MR513957}
L.~A. Caffarelli and N.~M. Rivi{\`e}re.
\newblock The {L}ipschitz character of the stress tensor, when twisting an
  elastic plastic bar.
\newblock \emph{Arch. Rational Mech. Anal.}, 69\penalty0 (1):\penalty0 31--36,
  1979.

\bibitem[Evans(1979)]{MR529814}
L.~C. Evans.
\newblock A second-order elliptic equation with gradient constraint.
\newblock \emph{Comm. Partial Differential Equations}, 4\penalty0 (5):\penalty0
  555--572, 1979.

\bibitem[Friedman(1982)]{MR679313}
A.~Friedman.
\newblock \emph{Variational Principles And Free-Boundary Problems}.
\newblock Pure and Applied Mathematics. John Wiley \& Sons, Inc., New York,
  1982.
\newblock A Wiley-Interscience Publication.

\bibitem[Gerhardt(1975)]{MR0385296}
C.~Gerhardt.
\newblock Regularity of solutions of nonlinear variational inequalities with a
  gradient bound as constraint.
\newblock \emph{Arch. Rational Mech. Anal.}, 58\penalty0 (4):\penalty0
  309--315, 1975.

\bibitem[Ishii and Koike(1983)]{MR693645}
H.~Ishii and S.~Koike.
\newblock Boundary regularity and uniqueness for an elliptic equation with
  gradient constraint.
\newblock \emph{Comm. Partial Differential Equations}, 8\penalty0 (4):\penalty0
  317--346, 1983.

\bibitem[Jensen(1983)]{MR697646}
R.~Jensen.
\newblock Regularity for elastoplastic type variational inequalities.
\newblock \emph{Indiana Univ. Math. J.}, 32\penalty0 (3):\penalty0 407--423,
  1983.

\bibitem[Lieberman(2001)]{MR1875900}
G.~M. Lieberman.
\newblock Regularity of solutions of obstacle problems for elliptic equations
  with oblique boundary conditions.
\newblock \emph{Pacific J. Math.}, 201\penalty0 (2):\penalty0 389--419, 2001.

\bibitem[Mariconda and Treu(2002)]{MR1881695}
C.~Mariconda and G.~Treu.
\newblock Gradient maximum principle for minima.
\newblock \emph{J. Optim. Theory Appl.}, 112\penalty0 (1):\penalty0 167--186,
  2002.

\bibitem[Rockafellar(1970)]{MR0274683}
R.~T. Rockafellar.
\newblock \emph{Convex Analysis}.
\newblock Princeton Mathematical Series, No. 28. Princeton University Press,
  Princeton, N.J., 1970.

\bibitem[Rozhkovskaya(1992)]{MR1205846}
T.~N. Rozhkovskaya.
\newblock Unilateral problems for elliptic systems with gradient constraints.
\newblock In \emph{Partial differential equations, {P}art 1, 2 ({W}arsaw,
  1990)}, volume~2 of \emph{Banach Center Publ., 27, Part 1}, pages 425--445.
  Polish Acad. Sci., Warsaw, 1992.

\bibitem[Treu and Vornicescu(2000)]{MR1797872}
G.~Treu and M.~Vornicescu.
\newblock On the equivalence of two variational problems.
\newblock \emph{Calc. Var. Partial Differential Equations}, 11\penalty0
  (3):\penalty0 307--319, 2000.

\bibitem[Wiegner(1981)]{MR607553}
M.~Wiegner.
\newblock The {$C\sp{1,1}$}-character of solutions of second order elliptic
  equations with gradient constraint.
\newblock \emph{Comm. Partial Differential Equations}, 6\penalty0 (3):\penalty0
  361--371, 1981.

\end{thebibliography}

\vspace{0.5cm}

\address{Department of Mathematics, UC Berkeley, Berkeley, CA 94720, USA}

\email{safdari@berkeley.edu}
\end{document}